\documentclass[12pt]{amsart}
\usepackage[margin=0.7in]{geometry}

\usepackage{amsmath}
\usepackage{amssymb}
\usepackage{latexsym}
\usepackage{enumerate}
\usepackage{dsfont}
\usepackage{graphicx}
\usepackage{verbatim}
\usepackage{mdwlist}
\usepackage{mathabx}
\usepackage{color}
\usepackage{cancel}
\usepackage{float}
\usepackage{ifpdf}
\ifpdf
\usepackage{hyperref}
\else
\usepackage[hypertex]{hyperref}
\fi

\newtheorem{theorem}{Theorem}[section]

\newtheorem{lemma}[theorem]{Lemma}
\newtheorem{corollary}[theorem]{Corollary}

\newenvironment{questionrm}{\begin{question} \rm}{ \end{question}}
\newtheorem{question}[theorem]{Question}
\newtheorem{defi}[theorem]{Definition}
\newenvironment{definition}{\begin{defi} \rm}{ \end{defi}}
\newtheorem{exa}[theorem]{Example}

\newtheorem{notice1}{Observation}
\newtheorem{rem}[theorem]{Remark}

\renewcommand{\phi}{\varphi}

\newcommand{\set}[1]{\{{#1}\}}

\newcommand{\ol}[1]{\overline{#1}}

\newcommand{\bs}[1]{\boldsymbol{#1}}

\newcommand{\st}{\operatorname{St}}

\newcommand{\ang}[1]{\ensuremath{\langle} #1 \ensuremath{\rangle}}

\newcommand{\tie}{^\frown}

\newcommand{\ocwait}{\ensuremath{\mathrm{wait}}}
\newcommand{\ocstop}{\ensuremath{\mathrm{stop}}}

\newcommand{\ocinf}{\mathrm{inf}}
\newcommand{\oci}[1]{\infty_{#1}}

\newcommand{\card}[1]{\text{card}({#1})}

\title[Nondensity of the $s$-degrees of the $\Sigma^0_2$ sets]{The singleton 
degrees of the $\Sigma^0_2$ sets are not dense}

\author[T.~F. Kent]{Thomas F. Kent}
\address{Independent Researcher, Iowa, 50010, USA}
\email{\href{mailto:tfkent.resarch@gmail.com}{tfkent.resarch@gmail.com}}
\thanks{The first author was partially supported by Marie 
Curie Incoming International Fellowship of the European Community 
FP6 Program under contract number MIFI-CT-2006-021702.}

\author[K.M.~Ng]{Keng Meng Ng}
\address{Division of Mathematical Sciences, School of Physical \& Mathematical
Sciences, College of Science\\
Nanyang Technological University\\
Singapore}
\email{\href{mailto:kmng@ntu.edu.sg}{kmng@ntu.edu.sg}}
\thanks{The second author was partially supported by the Ministry of 
Education, Singapore, under its Academic Research Fund Tier 2 (MOE-T2EP20222-0018).}

\author[A.~Sorbi]{Andrea Sorbi}
\address{Dipartimento di Ingegneria Informatica e Scienze Matematiche\\
Universit\`a Degli Studi di Siena\\
I-53100 Siena, Italy}
\email{\href{mailto:andrea.sorbi@unisi.it}{andrea.sorbi@unisi.it}}
\thanks{The third author is a member of INDAM-GNSAGA}



\begin{document}

\maketitle

\begin{abstract}
Answering an open question raised by Cooper, we show that there exist 
$\Delta^0_2$~sets $D$ and $E$ such that the  singleton degree of $E$ is a 
minimal cover of the singleton degree of $D$. This shows that the 
$\Sigma^{0}_{2}$ singleton degrees, and the $\Delta^{0}_{2}$ singleton 
degrees, are not dense (and consequently the $\Pi^0_2$ $Q$-degrees, and the 
$\Delta^{0}_{2}$ $Q$-degrees, are not dense). Moreover $D$ and $E$ can be 
built to lie in the same enumeration degree. 
\end{abstract}

\section{Introduction}
The importance of enumeration reducibility and its degree structure has 
considerably grown in recent years, in the wake of several striking results, 
such as the first order definability of the Turing degrees inside the 
enumeration degrees~\cite{Turing-definability}. Enumeration reducibility is a 
formalization of the notion of relative enumerability between sets of natural 
numbers, the informal idea being that a set $A$ is enumerable relatively to a 
set $B$ if there exists an algorithm which enumerates $A$ given any 
enumeration of $B$. Formally, given sets $A, B \subseteq \omega$ (where 
$\omega$ denotes the set of natural numbers) we say that $A$ is 
\emph{enumeration reducible} (or, simply, \emph{$e$-reducible}) to $B$ (in 
symbol: $A \leq_{e}B$) if there exists a computably enumerable (c.e.) set 
$\Phi$ such that 
\[
A=\set{x: (\exists u)[\ang{x, u} \in \Phi \, \&\, D_{u} \subseteq B]},
\]
where $D_{u}$ is the finite set with canonical index $u$. (In the following, 
we will often identify finite sets with their canonical indices, thus writing 
for instance $\ang{ x, F } \in \Phi$ meaning that $F$ is finite and $\ang{x, 
u } \in \Phi$ where $F=D_{u}$.) We also write in this case $A=\Phi(B)$, thus 
viewing any c.e.\ set as an operator (called an \emph{$e$-operator}) from 
sets to sets.  The structure of the \emph{$e$-degrees} is the set of 
equivalence classes of the equivalence relation $\equiv_e$ on sets, in which 
$A \equiv_e B$ if and only if $A\leq_e B$ and $B \leq_e A$. The $e$-degree of 
a set $A$ will be denoted by $\bs{a} = \deg_e(A)$. 

An important restricted version of $e$-reducibility is the so-called 
\emph{singleton enumeration reducibility}, or simply \emph{$s$-reducibility}, 
which only uses $e$-operators $\Phi$ subject to the restriction 
\[
\ang{ x, F } \in \Phi \Rightarrow [\textrm{$F=\emptyset$ or $F$
is a singleton}].
\]
Any $e$-operator satisfying this restriction is called an $s$-operator. We 
say that a set $A$ is \emph{$s$-reducible} to a set $B$ (denoted by 
$A\leq_{s}B$) if there is an $s$-operator $\Phi$ such that $A=\Phi(B)$. The 
equivalence relation originated by $\leq_{s}$ is denoted by $\equiv_{s}$ 
(where $A \equiv_{s} B$ if $A\leq_{s} B$ and $B\leq_{s} A$), and its 
equivalence classes are called \emph{$s$-degrees}: the $s$-degree of a set 
$A$ will be denoted by $\bs{a} = \deg_s(A)$. 

In general, for $r \in \{s,e\}$ we use the same symbol $\leq_{r}$ to denote 
both the reducibility (i.e. the pre-ordering relation on sets) and the 
partial ordering relation induced by the reducibility on the $r$-degrees. 

Although $s$-reducibility is strictly included in $e$-reducibility, in most 
concrete examples where $e$-reducibility shows up in mathematics or logic, it 
takes the form of $s$-reducibility. This is probably due to the fact that 
$\leq_{e}$ nicely embeds into $\leq_{s}$ by $A \leq_{e} B$ if and only if 
$A^{*}\leq_{s} B^{*}$, where, given a set $X$, we denote by~$X^{*}$ the set 
of all finite subsets of $X$, identified by their canonical indices. Interest 
in $s$-reducibility (often through its isomorphic presentation as 
$Q$-reducibility, of which we recall later the precise definition), relies 
also in its many applications to other sections of computability theory, to 
logic, and general mathematics. For instance it plays a key role in 
Marchenkov's solution of Post's Problem using Post's methods 
(\cite{Marchenkov:Incomplete}); and has applications to word problems (see 
for instance, \cite{Belegradek, Dobritsa}) and to abstract computational 
complexity (see for instance, \cite{Blum-Marques:Complexity}, 
\cite{Gill-Morris:subcreative}). For a survey on $Q$-reducibility and its 
applications see \cite{Omanadze:survey}. 

Throughout the rest of the paper we assume to have fixed some effective 
listing $\set{\Phi_{e}: e \in \omega}$ of the $s$-operators. Without loss of 
generality, we assume to work with uniformly computable approximations 
$\set{\Phi_{e, s}: e, s \in \omega}$ to the $s$-operators such that for every 
$e$ and $s$, $\Phi_{e, 0}= \emptyset$, $\Phi_{e, s+1} \smallsetminus \Phi_{e, 
s}$ contains at most one element, and if $\ang{x, F} \in \Phi_{e, s}$ then 
$x, \max (F) <s$ (where we understand $\max(\emptyset)=-1$). By 
\cite[Proposition~5]{Cooper-McEvoy} we may also assume that if $\set{X_s}_{s 
\in \omega}$ is any computable sequence of finite sets then we can combine 
this sequence with $\set{\Phi_{e, s}: s \in \omega}$, to get a computable 
sequence $\set{\Phi_{e}(X)[s]}_{s \in \omega}$ of finite sets, such that if 
$\set{X_s}_{s \in \omega}$ is a $\Sigma^0_2$-approximation to a set $X$ 
(i.e., $X=\set{x: (\exists t)(\forall s \geq t)[ x \in X_s] }$)), then 
$\set{\Phi_{e}(X)[s]}_{s \in \omega}$ is a $\Sigma^0_2$-approximation to the 
set $\Phi_e(X)$.

\section{A minimal cover in the $\Sigma^0_2$ $s$-degrees}

One of the most celebrated theorems on $e$-degrees is Gutteridge's theorem 
\cite{Gutteridge}, stating that there is no minimal $e$-degree. Gutteridge's 
proof works for the $s$-degrees as well. Thus (where we let $\bs{0}_s$ denote 
the least $s$-degree, consisting of all c.e. sets) 

\begin{theorem}\cite{Gutteridge}\label{theorem:gutteridge}
For every $s$-degree $\bs{a}>_s \bs{0}_s$, there exists an $s$-degree 
$\bs{b}$ such that $\bs{0}_s<_s \bs{b} <_s \bs{a}$. 
\end{theorem}

\begin{proof}
See \cite{Gutteridge}. See also \cite{Cooper:e-Density}.
\end{proof}

On the other hand, another celebrated theorem concerning the $e$-degrees (the 
Cooper Density Theorem~\cite{Cooper:Density-II}) reads that the $e$-degrees 
below the first jump $\bs{0}'_{e}$ are dense. (We recall that 
$\bs{0}'_{e}=\deg_{e}(\ol{K})$, where $\ol{K}$ is the complement of the 
halting set, and the $e$-degrees below $\bs{0}'_{e}$  partition exactly the 
$\Sigma^{0}_{2}$ sets.) Cooper's Density Theorem was later improved by 
Lachlan and Shore~\cite{Lachlan-Shore:Density} to show that density is shared 
by the $e$-degrees of any class of sets possessing what Lachlan and Shore 
call ``good'' approximations, see~\cite{Lachlan-Shore:Density}: the 
$\Sigma^0_2$ sets form such a class. 

Motivated by his theorem showing density of the $\Sigma^{0}_{2}$ $e$-degrees, 
Cooper (see \cite[Question~\S5(e)(11)]{Cooper:Enumeration}, and 
\cite[Question~26]{Cooper:Enumeration-Survey}) asked whether the 
$\Sigma^{0}_{2}$ $s$-degrees are dense as well. Our 
Theorem~\ref{thm:delta2emptyintervals} below answers that question, by 
showing in fact that there are $\Delta^{0}_{2}$ $s$-degrees $\bs{d}<_s 
\bs{e}$ such that the $s$-degree interval $(\bs{d}, \bs{e})$ is empty, i.e.~ 
$\bs{e}$ is a minimal cover of~$\bs{d}$ in the $s$-degrees.   

We recall that Downey, La~Forte and Nies \cite{Downey-Laforte-Nies:Q-degrees} 
showed that the c.e.\ $Q$-degrees are dense. A set $A$ is 
\emph{$Q$-reducible} to a set $B$ (in symbols: $A \leq_Q B$) if there exists 
a computable function $f$ such that 
\[
(\forall x)[x \in A \Leftrightarrow W_{f(x)}\subseteq B],
\]
where we refer to the standard listing $\{W_e\}_{e \in \omega}$ of the 
computably enumerable sets. Via the isomorphism between $Q$-degrees and 
$s$-degrees provided by complements of sets (if $B\ne \omega$, then $A\leq_Q 
B$ if and only if $\ol{A}\leq_{s}\ol{B}$, where for any given set $X 
\subseteq \omega$, the symbol $\ol{X}$ denotes the complement $\omega 
\smallsetminus X$ of that set\footnote{Indeed, originally Friedberg and 
Rogers~\cite{Friedberg-Rogers} defined $s$-reducibility by $A \leq_s B$ if 
there is a computable function $f$ such that $ (\forall\,x)[x\in A 
\Leftrightarrow W_{f(x)} \cap B \neq \emptyset]$: it is easy to see that if 
$B \neq \emptyset$ then this definition coincides with the definition of 
$\leq_s$ given earlier.}) we then have that the $\Pi^{0}_{1}$ $s$-degrees are 
dense. In view again of this isomorphism, 
Theorem~\ref{thm:delta2emptyintervals} shows that the $\Delta^{0}_{2}$ 
$Q$-degrees are not dense (in fact, there are $\Delta^{0}_{2}$ $Q$-degrees 
$\bs{a}<_Q \bs{b}$ such that the $Q$-degree interval $(\bs{a}, \bs{b})$ is 
empty), thus showing also that the density result by Downey, La~Forte and 
Nies cannot be improved from c.e.\ sets to $\Delta^{0}_{2}$ sets. 

\subsection{The theorem} The following theorem answers Cooper's question 
whether the $\Sigma^{0}_{2}$ $s$-degrees are dense. 

\begin{theorem}\label{thm:delta2emptyintervals}
There exist $\Delta^0_2$, in fact co-2-c.e., sets $D$ and $E$ such that $D 
<_s E$ and there is no set~$Z$ with $D <_{s} Z <_{s} E$. 
\end{theorem}

\begin{proof}
We wish to construct sets $A$ and $D$ satisfying the following requirements, 
for all numbers $k,e \in \omega$, 
\begin{align*}
R_k &:A\neq\Phi_k(D),\\
S_e &: Z_{e}=\Phi_{e}(A\oplus D)\Rightarrow
\left[ Z_{e}=\Gamma_e(D)~\vee~A=\Delta_e(Z_{e})\right],
\end{align*}
where $\Gamma_e$ and $\Delta_e$ are two $s$-operators built by us. 
Satisfaction of all $R$-requirements gives that $A\nleq_s D$, thus $D <_{s} 
E$ where $E=A \oplus D$. On the other hand, satisfaction of all 
$S$-requirements guarantees that there is no $Z$ such that $D <_{s} Z <_{s} 
E$.  This follows from the fact that, by the $S$-requirements, if $D \leq_{s} 
Z \leq_{s} A \oplus D$ then either we have $Z\leq_{s} D$ and thus 
$Z\equiv_{s} D$, or $A\leq_{s} Z$ and thus $A \oplus D \leq_{s} Z$ giving $A 
\oplus D \equiv_{s} Z$. 

\subsection{The strategies}

The construction is similar to the construction of an Ahmad pair in the 
$\Sigma^{0}_{2}$ e-degrees, see \cite{Ahmad:some-pairs}. We briefly outline 
the strategies. 

\subsubsection{The $R_{k}$-strategy} 
The strategy to meet the requirement $R_{k}$ picks a \emph{follower} (or, a 
\emph{witness})~$x \in A$; waits for $x$ to show up in $\Phi_{k}(D)$ via some 
axiom $\ang{ x, F} \in \Phi_k$ with $F \subseteq D$ (we say in this case that 
$x$ becomes \emph{realized}); at this points extracts $x$ from $A$, and 
restrains $x \in \Phi_{k}(D)$, by restraining $F \subseteq D$. 

\subsubsection{The $\Gamma$-strategy of the $S_{e}$-strategy}
The strategy to meet the requirement $S_e$ tries at first to build an 
$s$-operator $\Gamma_e$, such that $Z_e=\Gamma_e(D)$, as follows. For every 
$z\in Z_{e}\smallsetminus \Gamma_e(D)$ the strategy appoints an axiom $\ang{ 
z, F} \in \Gamma_{e}$, and tries to keep $\Gamma_{e}$ correct, making sure 
that $z$ gets extracted from $\Gamma_{e}(D)$ (by making $F\nsubseteq D$, and 
thus invalidating the axiom $\ang{ z, F} \in \Gamma_{e}$) if $z$ leaves 
$Z_{e}$. In more detail, if we see $z \in \Phi_{e}(\emptyset\oplus \emptyset) 
\smallsetminus \Gamma_e(D)$ then we enumerate the axiom $\ang{ z, \emptyset} 
\in \Gamma_{e}$: notice that in this case $z$ will never leave 
$\Phi_{e}(A\oplus D)$ and, consistently, the axiom which has been enumerated 
into $\Gamma_e$ permanently keeps $z \in \Gamma_{e}(D)$; if we see $z \in 
\Phi_{e}(\emptyset \oplus \set{d})\smallsetminus \Gamma_e(D)$ with $d \in D$ 
(we say in this case that $d$ is a \emph{$D$-use of $z$ in $\Phi_e$)} then we 
enumerate the axiom $\ang{ z, \set{d}} \in \Gamma_{e}$: notice that such an 
axiom does not require any correcting action because if $z$ leaves $Z_{e}$ 
due to $d$ leaving $D$ then the axiom $\ang{ z, \set{d}} \in \Gamma_{e}$ 
automatically becomes invalid; if we see $z \in \Phi_{e}(\set{x} \oplus 
\emptyset) \smallsetminus \Gamma_e(D)$ and $x \in A$ (we say in this case 
that $x$ is an \emph{$A$-use} of $z$ in $\Phi_e$) then we appoint a 
\emph{$\gamma_{e}$-marker $\gamma_{e}(z,x)=d \in D$} for the pair $(z,x)$, 
and enumerate the axiom $\ang{z, \set{d}}\in \Gamma_{e}$; subsequently, the 
$S_{e}$-strategy may want to correct the value of $\Gamma_{e}(D)(z)$ by 
extracting the $\gamma_e$-marker $\gamma_{e}(z,x)$ from $D$ if it sees that 
$z$ leaves $Z_{e}$ due to $x$ leaving $A$. The $\gamma_{e}$-markers are 
picked in an injective way: distinct pairs $(z,x)$ get distinct markers. The 
attempt of $S_e$ at satisfying itself, by building and maintaining a suitable 
$s$-operator $\Gamma_e$, will be called \emph{the $\Gamma$-strategy of 
$S_e$}. 

\subsubsection{$R_{k}$ below $S_{e}$} \label{sec:R-below-S}
Let us describe how the individual actions of the strategies for $S_{e}$ and 
$R_{k}$ interact with each other when $S_{e}$ has higher priority than 
$R_{k}$. As already explained, the $R_{k}$-strategy picks a witness $x \in A$ 
and waits for this witness to become realized. Waiting forever corresponds to 
a winning \emph{finitary} outcome for $R_k$ (since in this case $x \in A 
\smallsetminus \Phi_k(D)$), which does not injure the $\Gamma$-strategy of 
$S_e$, i.e.\ it does not interfere with the possibility for $S_e$ of 
maintaining a correct $\Gamma_e$. 

If at some stage we see $x\in \Phi_k(D)$, then the $R_{k}$-strategy would 
like to extract $x$ from~$A$ and restrain $x\in \Phi_k(D)$.  At this point, 
we have $x \in \Phi_k(D)$ thanks to some axiom $\ang{x, F} \in \Phi_k$ with 
$F \subseteq D$.  Since there might be more than one such set $F$, we pick 
the one such that $\ang{ x, F }$ has been responsible for keeping $x\in 
\Phi_k(D)$ for the longest time. As we assume that at most one element enters 
$\Phi_k$ at any given stage, this set is uniquely determined by $x$, and will 
be denoted by $F_x$.  Clearly, $F_x=\emptyset$ or $F_x=\set{d_x}$ for some 
number $d_x$). Ideally, the $R_{k}$ strategy would restrain $x\in \Phi_k(D)$ 
by restraining $F_x \subseteq D$. 

However, the extraction of $x$ from~$A$ may interfere with the $S_e$-strategy 
aiming at maintaining a correct $\Gamma_e$. Consider all numbers $z \in 
Z_{e}$ with $A$-use $x$ in $\Phi_e$, for which we have already appointed 
$\gamma_e$-markers $\gamma_e(z,x)$ which are still in $D$. If $R_k$ goes on 
with the extraction of $x$ from $A$, then $S_e$ will want to respond by 
extracting all these $\gamma_{e}$-markers, in order to correct 
$\Gamma_e(D)(z)$. If all these actions can be performed (extraction of $x$ 
from~$A$, restraint of $x\in \Phi_k(D)$ by restraining $F_x\subseteq D$, 
correction of $\Gamma_e(D)(z)$, for all relevant~$z$), including the lucky 
case when we have no correction to make, then we say that the pair $x$, $F_x$ 
is \emph{$\Gamma_e$-cleared} (for a rigorous definition, see 
Definition~\ref{def:i-cleared}). As~$F_x$ is uniquely determined by $x$, we 
might also simply say that $x$ is \emph{$\Gamma_e$-cleared}. With these 
actions, $R_{k}$ succeeds in winning $x\in \Phi_{k}(D)\smallsetminus A$, and 
$R_k$ takes a \emph{finitary} outcome which again does not interfere with the 
$\Gamma$-strategy of $S_e$. 

\ref{sec:R-below-S}.(a) \emph{Emergence of a setup: the $\Delta$-strategy of 
$S_e$.} However, when the witness $x$ becomes realized (through axiom 
$\ang{x,\set{d_x}} \in \Phi_k$ with $d_x \in D$), $R_k$ might have problems 
to restrain $d_{x}\in D$ to achieve $x \in \Phi_k(D)$, without injuring the 
$\Gamma$-strategy of $S_e$. Indeed, if $x$ is not $\Gamma_e$-cleared, there 
could be a number $z$ such that $d_{x}=\gamma_e(z,x)$, and 
$\Gamma_e$-correction (demanding the extraction of~$d_{x}$) would conflict 
with restraining $d_{x}\in D$. (Notice that by injectivity of the 
$\gamma_{e}$-markers, \emph{at most} one number $z$ will have 
$\gamma_{e}(z,x)=d_{x}$.) Since $x$ is not $\Gamma_e$-cleared, we cannot 
extract $x$ from $A$ without extracting $z$ from $\Phi_e(A\oplus D)$. So, for 
the pair $(z,x)$ we observe the correspondence ``$x\in A\Leftrightarrow z\in 
Z_{e}$'' which we describe by saying that $(z,x)$ is a 
\emph{$\Delta_e$-setup}. In this case, $R_k$ turns to the 
\emph{$\Delta$-strategy of $S_e$} and tries to use the setup for the purpose 
of extending the definition of an $s$-operator $\Delta_e$ such that 
$A=\Delta_e(Z_e)$: we put $x$ into $\Delta_e(Z_{e})$ with use $z$, by 
enumerating the axiom $\ang{x, \set{z}} \in \Delta_{e}$, and to keep 
$\Delta_{e}$ correct we wish that the $\Delta_e$-setup $(z,x)$ maintain its 
validity, i.e.\ the correspondence ``$x\in A\Leftrightarrow z\in Z_{e}$''; we 
add $x$ to the \emph{stream} of the $\Delta_{e}$-strategy (which is a set 
comprised of numbers $x'$ for which there is a $\Delta_e$-setup $(x',z')$ 
such that we have a valid correspondence ``$x'\in A\Leftrightarrow z'\in 
Z_{e}$''); we remove all $\gamma_{e}(\cdot, x)$-markers from $D$, (in fact, 
almost all of them, respecting higher priority restraints): we say in this 
case that we \emph{kill} $\Gamma_{e}$. If this injury process of $R_k$ 
against the $\Gamma$-strategy of $S_e$ is repeated infinitely many times, 
then the stream grows to be infinite. This allows the lower priority 
strategies to pick up their witnesses from the elements of the stream, 
without harming the correctness of $\Delta_e$ as long as the various setups, 
corresponding to the elements in the stream, remain valid. Note that if this 
is the true outcome, then $R_k$ takes an \emph{infinitary} outcome, under 
which it gives up its own satisfaction (this task must therefore be delegated 
to another strategy), turning its efforts to the satisfaction of~$S_e$, 
through the building of a correct $\Delta_e$. 

\ref{sec:R-below-S}.(b) \emph{Removal of a setup: going back to the 
$\Gamma$-strategy.} A $\Delta_e$-setup $(z,x)$ can later loose its validity. 
This is the case if a new, and currently valid, axiom for $z$ is enumerated 
into $\Phi_{e}$.  In this case, if this new axiom is of the form $\ang{z, 
\emptyset \oplus \emptyset} \in \Phi_{e}$, or has an $A$-use $a\in A$ in 
$\Phi_e$ different from~$x$, then the $R_{k}$-strategy can diagonalize by 
extracting $x$ from $A$ and (in the latter case) keeping $a$ in $A$.  This 
ensures that the value $\Phi_{e}(A\oplus D)(z)=1$ will not request any 
$\Gamma_{e}$-correction to extract~$z$ from $\Gamma_{e}(D)$, and so the 
extraction of $x$ from $A$ is compatible with having $d_{x} \in D$. If the 
use of $z$ in $\Phi_e$ in the new axiom is a $D$-use $d\in D$, and $d$ is 
currently not a $\gamma_{e}(\cdot, x)$-marker, then the $R_{k}$-strategy can 
diagonalize as above by keeping $d$ in~$D$ (after all, there is no need to 
extract $d$ from $D$ in the future since $d$ will not be a 
$\gamma_{e}$-marker which the $\Gamma$-strategy of $S_e$ needs to extract). 
In each of the situations described above, $R_k$ goes on with the extraction 
of $x$ from $A$, having $d_{x} \in D$, without interfering with 
$\Gamma_e$-correction. So $R_{k}$ succeeds in winning $x\in 
\Phi_{k}(D)\smallsetminus A$, and takes a \emph{finitary} outcome which does 
not injure the $\Gamma$-strategy of $S_e$. 

\ref{sec:R-below-S}.(c) \emph{Persistence of a setup in the limit.} However 
we might have a problem if a new axiom $\ang{z, \emptyset \oplus \set{d}} \in 
\Phi_e$ appears, and $d$ is a $\gamma_{e}(\cdot,x)$-marker: it could be that 
$d=\gamma_{e}(z', x)$ for some $z' \in Z_{e}$ with $A$-use $x$ in $\Phi_{e}$. 
Now, it might be problematic to switch back to the $\Gamma$-strategy for 
$S_e$ because the extraction of $x$ from $A$ for diagonalization would cause 
$z'$ to leave $Z_{e}$, and the resulting $\Gamma_{e}$-correction due to $z'$ 
leaving~$Z_{e}$ would require the extraction of $d=\gamma_{e}(z',x)$ from 
$D$. In other words, the extraction of $x$ from~$A$ must be accompanied by 
the extraction of $d$ from $D$ and so the $R_{k}$-strategy might not be able 
to diagonalize, even though the correspondence ``$x\in A\Leftrightarrow z \in 
Z_{e}$'' is false, and hence the $s$-operator $\Delta_{e}$ is incorrect. 
However, observe that the correctness of $\Delta_{e}$ and hence the 
correspondence ``$x\in A\Leftrightarrow z \in Z_{e}$'' are required to hold 
only if $R_k$ takes the infinitary outcome corresponding to the 
$\Delta$-strategy of $S_e$. If this outcome is the true outcome, then in fact 
all markers $\gamma_{e}(\cdot, x)$ will be extracted from $D$. In fact, if $z 
\in \Phi_{e}(A\oplus D)$ keeps getting new axioms in the $D$-side with 
$D$-uses which are markers of the form $\gamma_{e}(\cdot, x)$, then all of 
these axioms will be automatically killed every time $R_{k}$ takes the 
infinitary outcome. Hence at the end of time, the only possibly applicable 
$\Phi_{e}$-axiom for $z$ which is left is the axiom $\ang{z, \set{x}\oplus 
\emptyset} \in \Phi_{e}$, and so the correspondence ``$x\in A\Leftrightarrow 
z\in Z_{e}$'' is true in the limit after all, and the value 
$\Delta_e(Z_e)(z)$ is correct. 

\subsection{The tree of strategies}
The strategies used in the construction have outcomes that can be
conveniently described by the elements of the set $\Lambda = \set{\ocstop,
\ocwait, \ocinf}\cup \set{\oci{i}: i \in \omega}$ ordered by
\[
\ocstop <_{\Lambda} \oci{0} <_{\Lambda}
\oci{1} <_{\Lambda} \cdots <_{\Lambda} 
\oci{n} <_{\Lambda} \cdots <_{\Lambda}
\ocwait <_{\Lambda} \ocinf.
\]
By $\Lambda^{<\omega}$ we denote the set of finite strings of elements of 
$\Lambda$. We use standard notations and terminology about strings. In 
particular, if $\alpha, \beta \in \Lambda^{<\omega}$, then the symbol $\left| 
\alpha \right|$ denotes the length of $\alpha$; the symbol $\alpha <_{L} 
\beta$ denotes that $\alpha$ is \emph{to the left} of $\beta$, i.e. there 
exists $i$ such that the $i$th component $\alpha(i)$ of $\alpha$ is different 
from the $i$th component $\beta(i)$ of $\beta$, and for the least such $i$ we 
have $\alpha(i)<_{\Lambda} \beta(i)$. The symbol $\alpha\leq \beta$ denotes 
$\alpha<_{L} \beta$, or $\alpha$ is an \emph{initial segment of $\beta$} 
(denoted by $\alpha \subseteq \beta$; if $\alpha \subseteq \beta$ and $\alpha 
\ne \beta$ then we write $\alpha \subset \beta$); the symbol $\alpha< \beta$ 
denotes $\alpha \leq \beta$ but $\alpha \ne \beta$. If $\alpha < \beta$ then 
we say that $\alpha$ has \emph{higher priority} than $\beta$ (or $\beta$ has 
\emph{lower priority} than $\alpha$). If $g$ is an infinite path through 
$\Lambda^{<\omega}$ and $n$ is a number, then $g\!\upharpoonright\! n$ 
denotes the initial segment of $g$ having length $n$. 

We now define the \emph{tree of strategies} $T \subseteq \Lambda^{<\omega}$ 
as follows. We effectively order the $R$- and $S$-requirements with the same 
order type as $\omega$, and effectively list them as $\set{U_e: e \in 
\omega}$: we say that $U_i$ has \emph{higher priority} than $U_e$ (or $U_e$ 
has \emph{lower priority} than $U_i$) if $i<e$. By induction on the length of 
the strings, we define which elements of $\Lambda^{<\omega}$ belong to $T$, 
the way they are assigned requirements, and whether or not they are active or 
satisfied  along their extensions. If $U$ is a requirement, we will say that 
a string $\beta$ is a \emph{$U$-string} if it has been assigned the 
requirement $U$. In particular, an \emph{$R$-string} is a string which has 
been assigned an $R$-requirement, and an \emph{$S$-string} is a string which 
has been assigned an $S$-requirement. The elements of $T$ will be simply 
called \emph{strings} (as we do not care about strings in 
$\Lambda^{<\omega}\smallsetminus T$), or also \emph{nodes}, or 
\emph{strategies}. 

Let the empty string $\lambda$ lie in $T$, and let $U_0$ be assigned to 
$\lambda$. Assume that we have already defined that $\alpha \in T$, and 
$\alpha$ has been assigned a requirement. We now define by cases the 
successors of $\alpha$, i.e., the elements $\beta \in T$ with 
$|\beta|=|\alpha|+1$:

\emph{Case 1}:  $\alpha$ is an $S$-strategy. Then $\alpha \tie \ang{\ocinf}$ 
is the immediate successor of $\alpha$. We will say that $\alpha$ is 
\emph{active along $\alpha \tie \ang{\ocinf}$}. Furthermore, for every $\beta 
\subset \alpha$, if $\beta$ is active or satisfied along $\alpha$ then 
$\beta$ is \emph{active} or \emph{satisfied along $\alpha \tie 
\ang{\ocinf}$}. The idea here is that if $\alpha$ is active, it is actively 
trying to satisfy its requirement by building and correcting a suitable 
$s$-operator $\Gamma$ (at which it may fail only due to the action of a lower 
priority $R$-strategy killing $\Gamma$ and building a different $s$-operator 
$\Delta$). Assign to $\alpha \tie \ang{\ocinf}$ the highest priority $U$ such 
that $U$ is not assigned to any $\beta$ which is either active or satisfied 
along $\alpha \tie \ang{\ocinf}$. 

\emph{Case 2:} $\alpha$ is an $R$-strategy. Then let $\set{\beta_{i}: i< n}$, 
with $\beta_0 \subset \beta_1 \subset \dots \subset \beta_{n-1} \subset 
\alpha$, be the set of $S$-strategies above $\alpha$ which are active along 
$\alpha$  (allowing $n = 0$). We let the immediate successors of $\alpha$ be 
$\alpha \tie \ang{o}$ for $o \in \set{\ocstop, \oci{0}, \dots, \oci{n-1}, 
\ocwait}$.  For $o \in \set{\ocstop, \ocwait}$, we say that $\alpha$ is 
\emph{satisfied along $\alpha \tie \ang{o}$}, and for every $\beta \subset 
\alpha$, if $\beta$ is active or satisfied along $\alpha$ then we say that 
$\beta$ is \emph{active} or \emph{satisfied along $\alpha \tie \ang{o}$}. Fix 
$i < n$: we say that $\alpha$ is \emph{neither active nor satisfied along 
$\alpha \tie \ang{\oci{i}}$}; for all $j < i$ we say that $\beta_j$ is 
\emph{active along $\alpha \tie \ang{\oci{i}}$}; we say that $\beta_i$ is 
\emph{satisfied along $\alpha \tie \ang{\oci{i}}$}; for all $j > i$ we say 
that $\beta_j$ is \emph{neither active nor satisfied along $\alpha \tie 
\ang{\oci{i}}$}; finally all other strategies are \emph{active} or 
\emph{satisfied along $\alpha \tie \ang{\oci{i}}$} if they are so along 
$\alpha$. For every outcome $o$ of $\alpha$, we assign to $\alpha \tie 
\ang{o}$ the highest priority $U$ such that $U$ is not assigned to any 
$\beta$ which is either active or satisfied along $\alpha \tie \ang{o}$. The 
intuition here is that an $R$-strategy $\alpha$ can infinitely injure the 
$\Gamma$-strategy of some $\beta_{j}$ as described above, and in the process 
it fails to meet its own requirement. In this case, strategy $\alpha$ 
isolates the strategy $\beta_{i}$ of highest priority for which this happens, 
and switches from the $\Gamma$-strategy to the $\Delta$-strategy for 
$\beta_i$ (as described in Section~\ref{sec:R-below-S}.(a)). Moreover, the 
$s$-operators of all $S$-strategies $\beta_{j}$, $i \leq j<n$, are destroyed, 
these strategies cease to be active, and thus the requirements corresponding 
to the nodes $\beta_{j}$, $i < j<n$, and $\alpha$, will need to be reassigned 
to nodes having lower priority in the priority tree. 

If $\alpha$ is an $R_{k}$-strategy then we may also write $\Phi_\alpha$ for 
$\Phi_k$; likewise, if $\alpha$ is an $S_{e}$-node then we may write 
$\Phi_\alpha$ for $\Phi_e$, and we may write $\Gamma_{\alpha}$ to represent 
the $s$-operator built by $\alpha$. If $\alpha$ is an $R$-node, and 
$\beta_i\subset \alpha$ is an $S$-node which is active along $\alpha$, then 
we may simply denote by $\Delta_{\beta_i}$ (instead of $\Delta_{\alpha \tie 
\ang{\oci{i}}}$), the $s$ operator built by $\alpha$ when turning to the 
$\Delta$-strategy of $\beta_i$. 

\subsubsection{The tree of strategies, and more on how the strategies 
interact}\label{ssct:more-on-stragies} We now try and give a somewhat more 
detailed description, using the tree of strategies, of how an $R$-strategy 
$\alpha$ interacts with the active $S$-strategies $\beta_0 \subset \beta_1 
\subset \dots \subset \beta_{n-1} \subset \alpha$ above it. Suppose that 
$\alpha$ is assigned the requirement $R_k$. It is just the case to observe 
that strategy $\alpha$ knows how to deal with the $S$-strategies 
$\beta\subset \alpha$ which are not active along $\alpha$, since the 
$\Gamma$-operators relative to these strategies are destroyed and almost all 
the related $\gamma$-markers are out of $D$. 

\emph{Outcome $\ocwait$}. If the last appointed witness is unrealized then 
$\alpha$ momentarily takes outcome $\ocwait$:  if this is eventually the true 
outcome then $\alpha$ wins the requirement $R_k$ since $x \in A 
\smallsetminus \Phi_\alpha(D)$. 

\emph{Outcome $\ocstop$}. Otherwise, all still \emph{uncanceled} (see 
Definition~\ref{def:aux}) witnesses are realized (henceforth the still 
uncanceled witnesses of $\alpha$, will be called the \emph{still available 
witnesses for $\alpha$}). If $\alpha$ sees some such witness $x$ which is 
$\Gamma_{\beta_i}$-cleared for all $i<n$, then extracts $x$ from $A$; 
restrains $F_x\subseteq D$ (by putting back $d_x\in D$ if needed: here 
$F_x=\set{d_x}$ is the finite set responsible for having made $x$ realized, 
as explained at the beginning of Section~\ref{sec:R-below-S}); permanently 
takes outcome $\ocstop$; and \emph{initializes} (see 
Definition~\ref{def:aux}) all lower priority strategies. 

Otherwise, all still available witnesses for $\alpha$ are realized, but none 
of them is $\Gamma_{\beta_i}$-cleared for all $i<n$. We distinguish the 
following cases. 

\ref{ssct:more-on-stragies}.(a) (\emph{One $S$-strategy $\beta_i \subset 
\alpha$}) Let us consider at first only the case of one $\beta_i \subset 
\alpha$, as in the above Section~\ref{sec:R-below-S}. Let us simply write 
$\beta=\beta_i$, and $\infty=\oci{i}$. Since no still available witness for 
$\alpha$ is $\Gamma_\beta$-cleared, then some new $\Delta_{\beta}$-setup 
$(z,x)$ has emerged (here $x$ is the last appointed witness of $\alpha$, 
which has now become realized), and it will now be exploited by $\alpha$ to 
promote the $\Delta$-strategy of $\beta$ (as described in 
Section~\ref{sec:R-below-S}.(a)), in particular $\alpha$ takes outcome 
$\infty$, and the number $x$ is added to the stream of $\alpha \tie 
\ang{\infty}$. If $\infty$ is the true outcome, then $\alpha \tie 
\ang{\infty}$ builds an infinite stream, it does not meet its requirement 
(because the elements of the stream may be moved in or out of $A$ by lower 
priority strategies), it kills $\Gamma_{\beta}$ (this solves the problems 
outlined in Section~\ref{sec:R-below-S}.(c) when analyzing the case of the 
setup $(z,x)$ being possibly endangered by the appearance of an axiom 
$\ang{z, \emptyset \oplus \set{d}} \in \Phi_e$ with $d \in D$, and $d$ being 
a $\gamma_\beta(\cdot,x)$-marker) but it builds a correct $\Delta_{\beta}$. 
The requirements for $S_{\beta}$ and $R_\alpha$ will need to be reassigned 
anew to nodes below $\alpha \tie \ang{\infty}$. 

\ref{ssct:more-on-stragies}.(b) (\emph{Two $S$-strategies $\beta_0 \subset 
\beta_1 \subset \alpha$}) Let us consider now the case of two $S$-strategies 
$\beta_0 \subset \beta_1 \subset \alpha$ which are active along $\alpha$. 
Since no still available witness for $\alpha$ is both $\Gamma_{\beta_0}$-, 
and $\Gamma_{\beta_1}$-cleared, then there exists exactly one $i \in 
\set{0,1}$ such that some new $\Delta_{\beta_i}$-setup $(z,x)$ has emerged 
(again here $x$ is the last appointed witness which has now become realized; 
$i$ is unique because $d_x$ cannot be both a 
$\gamma_{\beta_0}(\cdot,x)$-marker and a $\gamma_{\beta_1}(\cdot,x)$-marker). 

\begin{enumerate}
  \item[(i)] If $(z,x)$ is a $\Delta_{\beta_1}$-setup then the situation is 
      similar to \ref{ssct:more-on-stragies}.(a) and $\alpha$ takes outcome 
      $\oci{1}$, kills $\Gamma_{\beta_1}$ but does not injure the 
      $\Gamma$-strategy of $\beta_0$, and builds $\Delta_{\beta_1}$. 
      Strategy $\beta_1$ ceases to be active along $\alpha \tie 
      \ang{\oci{1}}$, and, together with $R_\alpha$, the corresponding 
      requirement will need to be reassigned anew to nodes below $\alpha 
      \tie \ang{\oci{1}}$. 

  \item[(ii)] If $(z,x)$ is a $\Delta_{\beta_0}$-setup then $\alpha$ takes 
      outcome $\oci{0}$, it exploits (as in section 
      \ref{sec:R-below-S}.(a)) the setup $(z_0,x)$ to build 
      $\Delta_{\beta_0}$, it kills both $\Gamma_{\beta_0}$ and 
      $\Gamma_{\beta_1}$, and, through initialization, cancels 
      $\Delta_{\beta_1}$ as well. Strategies $\beta_0$ and $\beta_1$ cease 
      to be active along $\alpha \tie \ang{\oci{0}}$, and, together with 
      $R_\alpha$, the corresponding requirements will need to be reassigned 
      anew to nodes below $\alpha \tie \ang{\oci{0}}$. 
\end{enumerate}

\ref{ssct:more-on-stragies}.(c) (\emph{The general case}) In the general case 
of several $S$-strategies $\beta_0 \subset \beta_1 \subset \dots \subset 
\beta_{n-1} \subset \alpha$, which are active along $\alpha$, the situation 
is similar to Case~\ref{ssct:more-on-stragies}.(b)(ii). Strategy $\alpha$ 
picks the unique~$\beta_i$ for which a new $\Delta_{\beta_i}$-setup emerges, 
and it now acts towards $\beta_i$ similarly to how $\alpha$ acts 
towards~$\beta_0$ as we have described in 
\ref{ssct:more-on-stragies}.(b)(ii); it acts towards each $\beta_j$ with 
$j>i$ similarly to how $\alpha$ acts towards $\beta_1$ as described above in 
\ref{ssct:more-on-stragies}.(b)(ii); it does not injure the $\Gamma$-strategy 
of any $\beta_j$ with $j<i$, similarly to the way $\alpha$ does not injure 
the $\Gamma$-strategy of $\beta_0$, as we have described above in 
\ref{ssct:more-on-stragies}.(b)(i). 

\subsection{The construction}

The construction is in stages. At the end of each stage $s$ we will have 
defined a cofinite approximation $A_s$ to the final desired set $A$, which 
will be eventually defined as 
\[
A=\set{x: \left(\exists t \right)\left(\forall s \geq t \right)[x \in A_s]},
\]
and a cofinite approximation $D_s$ to the final desired set $D$,  which will 
be eventually defined as 
\[
D=\set{x: \left(\exists t\right)\left(\forall s \geq t\right)[x \in D_s]}.
\]

Each stage $s>0$ consists of substages $t$, with $t\leq u$ for some $u \leq 
s$. At the end of each substage $t$ of stage $s$, we appoint some string 
$\alpha_{s,t}\in T$, with $|\alpha_{s,t}| = t$, which extends the strings 
appointed at the previous substages of stage $s$.  If $t=s$ or $\alpha_{t-1}$ 
has stopped the stage $s$ after appointing $\alpha_{t}$, then we go to stage 
$s+1$. At the end of stage $s$ we define $f_s$ to be the string appointed at 
the end of the last substage of stage $s$. We say that a stage $s$ is 
\emph{$\alpha$-true} if $\alpha \subseteq f_s$, i.e.\ $\alpha$ is appointed 
at some substage of stage $s$. 

\emph{The parameters and their approximations.} Each node $\alpha$ has its 
own parameters, which may include: a \emph{current witness} $x_\alpha$; 
$s$-operators $\Delta_\alpha$, $\Gamma_\alpha$; a set of numbers, called the 
\emph{stream of $\alpha$} and denoted by $\st(\alpha)$; and 
$\gamma_\alpha$-markers. Within any given stage $s$, the parameters are 
updated substage by substage, so that they get a corresponding approximation 
at the end of each substage $t$. If $s> 0$ is any stage then at the beginning 
of substage $t=0$ of stage $s$ the various sets, parameters, or operators, 
together with the current values of $A$ and $D$, are understood to be 
evaluated as at the end of stage $s-1$ (i.e., by taking the approximations 
defined at the end of the last substage of stage $s-1$); at the beginning of 
a substage $t>0$ of stage $s$, they are understood to be evaluated as at the 
end of substage $t-1$. We \emph{warn} the reader that throughout the 
construction we will be very lazy in specifying stages or substages at which 
the various parameters are approximated, confident that this will be clear 
from the context. Thus for instance, at the beginning of a substage $t>0$, by 
$\Phi_e(A \oplus D)$ we denote the current approximation to this set under 
the conventions above specified, meaning in fact the set 
$\Phi_{e,s}(A_{s,t-1} \oplus D_{s,t-1})$, where $A_{s,t-1}$ and $D_{s,t-1}$ 
are the values of $A$ and $D$ at the end of substage $t-1$ of stage $s$ (with 
$A_{s,0}=A_{s-1}$ if $s>0$, and $A_{0,0}=\omega$, and similarly, 
$D_{s,0}=D_{s-1}$ if $s>0$, and $D_{0,0}=\omega$). Finally, we define 
$A_s=A_{s,u}$ and $D_s=D_{s,u}$, where $u$ is the last substage of stage $s$. 

\begin{definition}\label{def:aux}[\emph{New numbers}, 
\emph{initialization}, \emph{cancellation}, \emph{suitable witnesses}] 
\begin{itemize}
\item A number is \emph{new at stage $s$} if it is greater than any number 
    mentioned so far in the construction. 

\item To \emph{initialize a node $\alpha$ at the end of stage $s$} means to 
    \emph{cancel} the values of all its parameters, which thus start again 
    at next stage as undefined, or (if they are $s$-operators or the stream  
    $\st(\alpha)$) as empty sets. At the end of stage $0$, all nodes are 
    initialized. At the end of stage $s$ all strategies $\alpha > f_s$ are 
    initialized. If an $R$-node $\alpha$ is initialized then all the 
    witnesses which have been chosen for $\alpha$ before this 
    initialization are \emph{canceled} (meaning that their current values 
    are cancelled). 

\item For $s>0$ define $p^\alpha_s=\max (|\alpha|, u)$, where $u < s$ is 
    the last stage at which~$\alpha$ has been initialized.  We say that a 
    number $x$ is \emph{suitable to be picked as a witness by an $R$-node 
    $\alpha$ at stage $s$} if the following conditions hold: $x \in A \cap 
    \st(\alpha)$ and $x \geq p_s^{\alpha \tie \ang{\ocwait}}$. Thus, the 
    following hold: $x > \left|\alpha \right|$; $x$ is greater than any 
    number in any stream $\st(\alpha\tie \ang{o})$ as it was defined before 
    $\alpha \tie \ang{\ocwait}$ was last initialized; and consequently, $x$ 
    is greater than any witness previously picked by $\alpha$. We also 
    observe that at each stage $s$, we have $\st_s(\alpha)\subseteq 
    [p^\alpha_s, s)$, where following the usual interval notation we denote 
    $[u,v)=\set{z: u\leq z < v}$. 
	\end{itemize}

\end{definition}

\emph{Stage $s = 0$}: Let $A_{0}=\omega$ and $D_{0}=\omega$. Initialize all 
$\alpha \in T$. 

\emph{Stage $s > 0$} and \emph{substage $t \leq s$}.

If $t=s$ then go to stage $s+1$.

If $t=0$ then let $\alpha_{s,0}=\lambda$, and $\st_{s,0}(\lambda) = [0, s)$, 
where $[0,s)=\set{x: x <s}$. 

If $0<t<s$ then assume that  $\alpha_{s,t-1}$ is the string appointed at 
substage $t-1$: for simplicity we will write $\alpha= \alpha_{s,t-1}$ and 
$\alpha_t = \alpha_{s, t}$. Assume also that the stage has not been stopped 
after defining~$\alpha$. We are now going to appoint $\alpha_t$. Unless 
otherwise specified, at the end of substage $t$ we will set 
$\st(\alpha_t)=\st(\alpha)\cap [p^{\alpha_t}_s, s)$. 

\underline{Case 1}: $\alpha$ is an $S_{e}$-strategy. This strategy consists 
in \emph{maintaining $\Gamma_\alpha$}. For each $\ang{z, F \oplus G} \in 
\Phi_e$ do the following: If $F = \emptyset$, enumerate the axiom $\ang{z, 
G}$ into $\Gamma$.  Otherwise $F = \set{x}$ for some $x$.  If $x \in A$, $x < 
s$, and the $\gamma$-marker $\gamma_\alpha(z, x)$ is not defined, define  
$\gamma_\alpha(z, x)$ to be a new number, and enumerate the axiom $\ang{z, 
\set{\gamma_\alpha(z, x)}}$ into $\Gamma$.  If $x \notin A$ and the 
$\gamma$-marker $\gamma_\alpha(z, x)$ is defined, extract $\gamma_\alpha(z, 
x)$ from~$D$, and cancel $\gamma_\alpha(z, x)$. 

\underline{Case 2}: $\alpha$ is an $R_{k}$-strategy. Let $\set{\beta_i: 
i<n}$, with $n \geq 0$, be the set consisting of all the $S$-strategies 
$\beta_0 \subset \beta_1 \subset \dots \subset \beta_{n-1} \subset \alpha$ 
that are active along $\alpha$. For each $i \in \set{0, \ldots, n-1}$ we 
let~$\Gamma_i$ denote the $s$-operator $\Gamma_{\beta_i}$ which is 
constructed by $\beta_i$, $\Phi_i$ represent~$\Phi_{\beta_i}$, and the 
symbols $\gamma_i(\cdot,\cdot)$, instead of $\gamma_{\beta_i}(\cdot,\cdot)$, 
denote the $\gamma$-markers appointed by $\beta_i$. Finally, we write 
$\Delta_i$ to represent the $s$-operator built by $\alpha$ when taking the 
outcome $\oci{i}$ towards satisfaction of the $S$-requirement associated 
with~$\beta_i$, via the $\Delta$-strategy of $\beta_i$. 

\underline{Case 2.1}: $\alpha$ has \emph{stopped}, i.e.\ there have been 
previous $\alpha$-true stages, at the last $\alpha$-true stage~$s^-$ we had 
$\alpha \tie \ang{\ocstop} \subseteq f_{s^-}$, and $\alpha$ has not been 
initialized since then. End the substage, and set $\alpha_t=\alpha \tie 
\ang{\ocstop}$. 

\underline{Case 2.2}: Otherwise. We further distinguish the following 
subcases: 

\underline{Case 2.2.1}: $x_\alpha$ is undefined. Define $x_\alpha$ to be the 
least number (if there is any such number) which is suitable to be picked as 
a witness by $\alpha$ at stage $s$ (see Definition~\ref{def:aux}): thus 
$x_\alpha$ becomes the new \emph{current witness}. Whether or not a new 
$x_\alpha$ is appointed, stop the stage, and set $f_s = \alpha_t = \alpha 
\tie \ang{\ocwait}$. 

\underline{Case 2.2.2}: $x_{\alpha}$ is defined and $x_\alpha \in A 
\smallsetminus \Phi_k(D)$. End the substage, and set $\alpha_t=\alpha \tie 
\ang{\ocwait}$. 

\underline{Case 2.2.3}: Otherwise $x_\alpha$ is defined and $x_\alpha \in A 
\cap \Phi_k(D)$ due to some axiom $\ang{x_\alpha, F} \in \Phi_k$, with $F 
\subseteq D$. We call $x_{\alpha}$ a \emph{realized} witness. Pick the finite 
set $F \subseteq D$ such that $\ang{x_\alpha,  F} \in \Phi_k$ and $F$ has 
been a subset of $D$ for the longest time (i.e. among such finite sets, $F$  
has least stage $u\leq s$ such that $F\subseteq D_t$ for every $u\leq t \leq 
s$), and call it $F_{x_\alpha}$. We distinguish further subcases depending on 
$n$: 

\begin{definition}\label{def:i-cleared}
For each $x$ and $i < n$, let $\widehat{D}_{i, x} \subseteq D$ denote the set 
of $\gamma$-markers $\gamma_i(z, x)$ appointed by $\beta_i$ since its last 
initialization which are still defined, and let $\widehat{D}_x = \bigcup_{i < 
n}\widehat{D}_{i, x}$.  For each $d \in \widehat{D}_{i, x}$, let $z_{i, x}$ 
denote the unique $z$ such that $d = \gamma_i(z_{i, x}, x)$. We say that 
$x_\alpha$, as defined in Case 2.2.3, is \emph{$\Gamma_i$-cleared} if 
$F_{x_\alpha} = \emptyset$ or $F_{x_\alpha} = \set{d_{x_\alpha}}$ and either 
$d_{x_\alpha} \notin \widehat{D}_{i, x_\alpha}$ or $d_{x_\alpha} = 
\gamma_i(z_{i,x_\alpha}, x_\alpha) \in \widehat{D}_{i, x_\alpha}$ and $z_{i, 
x_\alpha} \in \Phi_i((A \smallsetminus \set{x_\alpha}) \oplus ((D 
\smallsetminus \widehat{D}_{x_\alpha}) \cup F_{x_\alpha}))$. 
\end{definition}

\underline{Case 2.2.3.1}: $n = 0$ or $x_\alpha$ is $\Gamma_i$-cleared for all 
$i < n$.  Extract $x_\alpha$ from $A$; if $F_{x_\alpha} = \set{d_{x_\alpha}}$ 
and $d_{x_\alpha} = \gamma_i(z_{i, x_\alpha}, x_\alpha)$ for some $i < n$, 
necessarily unique, then: cancel $\gamma_i(z_{i, x_\alpha}, x_\alpha)$; stop 
the stage; and set $f_s = \alpha_t = \alpha \tie \ang{\ocstop}$. 

\underline{Case 2.2.3.2}:  $n > 0$ and there is an $i < n$ and a 
witness $x \in \st(\alpha\tie \ang{\oci{i}})$ which is $\Gamma_i$-cleared.  
As necessary, do each of the following: set $x_\alpha = x$ (canceling the 
currently defined $x_\alpha$ if necessary); extract $x$ from $A$; cancel 
$d_x$ (as defined in Case 2.2.3.3) as a $\gamma_i$-marker; set $F_{x_\alpha} 
= \set{d_x}$; enumerate $d_x$ into $D$; stop the stage; and set $f_s = 
\alpha_t = \alpha \tie \ang{\ocstop}$. 

\underline{Case 2.2.3.3}: Otherwise, $n > 0$ and there exists an $i < n$, 
necessarily unique, such that $x_\alpha$ is not $\Gamma_i$-cleared. Then 
$F_{x_\alpha} = \set{d_{x_\alpha}}$ where $d_{x_\alpha} = \gamma_i(z_{i, 
x_\alpha}, x_\alpha)$ for some $\gamma_i$-marker.  For simplicity, we denote 
$x = x_\alpha$, $d_{x} = \gamma_i(z_x, x)$, and $z_x = z_{i, x_\alpha}$.  We 
call the pair $(z_x, x)$ a \emph{$\Delta_i$-setup} at $s$.  
\begin{enumerate}
\item[(a)] Add $x$ to $\st(\alpha\tie \ang{\oci{i}})$ and the axiom 
    $\ang{x, \set{z_x}}$ to $\Delta_{i}$; 

\item[(b)] (\emph{killing $\Gamma_j$}) For all $j \geq i$, extract from $D$ 
    and cancel all $\gamma_j$-markers $d = \gamma_j(z, x)$ with $x > 
    |\alpha|$ which were picked since the most recent stage at which either 
    $\alpha\tie\ang{\oci{i}}$ was last active or last initialized; 

\item[(c)] For all $y \in A \!\upharpoonright\! s \smallsetminus S(\alpha 
    \tie \ang{\oci{i}})$, enumerate $\ang{y, \emptyset} \in \Delta_i$ 
    (where, for any given $X \subseteq \omega$ and $x \in \omega$, we let 
    $X \!\upharpoonright\! x=X \cap [0,x)$); 

\item[(d)] Cancel $x_\alpha$;

\item[(e)] end the substage, and set $\alpha_t=\alpha \tie \ang{\oci{i}}$.
\end{enumerate}

Notice that $x_\alpha > \max\left(\st_{s-1}(\alpha \tie 
\ang{\oci{i}})\right)$ since $x_\alpha$ is greater than all previous 
witnesses appointed by $\alpha$ (as observed in Definition~\ref{def:aux}), 
and $\alpha \tie \ang{\oci{i}}$ consists of witnesses previously appointed by 
$\alpha$. 

At the end of each stage $s$, \emph{initialize} all $\beta \in T$ having
lower priority than $f_s$.

\subsection{Verification}
It easily follows from the construction that if $\alpha \in T$ and $s$ is an 
$\alpha$-true stage then $\max(\st_s(\alpha))<s$. Indeed, (letting $\alpha^-$ 
denote the predecessor of $\alpha$, along $\alpha$) it is easy to see that 
$\st_s(\alpha) \subseteq \st_s(\alpha^-)$, and thus, by induction, 
$\st_s(\alpha) \subseteq \st_s(\lambda)$.

\begin{lemma}[Tree Lemma]\label{lem:tree-lemma}
\begin{enumerate}

\item Let $g$ be an infinite path through $T$. Then for every requirement 
    $U_e$, there exists a node $\alpha \subset g$ such that $\alpha$ is 
    assigned requirement $U_e$ and for every node $\beta$ with $\alpha 
    \subset \beta \subset g$, $\beta$ is not assigned requirement $U_e$ and 
    there exists a node $\delta$, with $\alpha \subseteq \delta \subset g$, 
    such that $\alpha$ is either active along all $\beta$ with $\delta 
    \subset \beta \subset g$, or satisfied along all such $\beta$. 

\item (Existence of True Path) There exists an infinite path $f$ through 
    $T$ (called the \emph{true path} of the construction) such that for 
    every $n$, the following hold:
     
\begin{enumerate}
\item There exists a last stage $t_n$ at which $f\!\upharpoonright\! n$ 
    is initialized (as a consequence we have that 
    $p^{f\!\upharpoonright\! n}_{fin}=\lim_s p^{f\!\upharpoonright\! 
    n}_s=\max\set{t_n, n}$); 

\item If a number $x$ is enumerated into $\st(f\!\upharpoonright\! n)$ 
    after $t_n$ then $x$ never leaves $\st(f\!\upharpoonright\! n)$ 
    thereafter; therefore the limit value of $\st(f\!\upharpoonright\! 
    n)$ (i.e. the set of numbers $x$ which have been enumerated after 
    $t_n$) is a c.e.\ (in fact, decidable) set, which is included in 
    $[p^{f\!\upharpoonright\! n}_{fin}, \infty)$; 

\item At cofinitely many $f\!\upharpoonright\! n$-true stage $s > t_n$ a 
    still available element $x> \max(\st_{s-1}(f\!\upharpoonright\! n))$ 
    (assume $\max (\emptyset)=-1)$) is enumerated into 
    $\st(f\!\upharpoonright\! n)$. 
 \end{enumerate}

\end{enumerate}
\end{lemma}

\begin{proof}
\begin{enumerate}

\item The proof is by induction on the priority ranking of the 
    requirements. 

\item It is enough to exhibit a set of strings $\set{\alpha_n: n\in 
    \omega}$ and a set $\set{t_n: n \in \omega}$ of stages such that the 
    following hold, for every $n$:
     
\begin{itemize}
    \item[(i)] $|\alpha_n| = n$ and $\alpha_n \subset \alpha_{n+1}$ for 
        every $n$; 

    \item[(ii)] there are infinitely many $\alpha_n$-stages; 
    
    \item[(iii)] $t_n$ is the last stage at which $\alpha_n$ is 
        initialized; 

    \item[(iv)] at cofinitely many $\alpha_n$-true stages $s > t_n$ we 
        have that a number $x \geq p^{\alpha_n}_{fin}$ enters 
        $\st(\alpha_n)$ such that  $x > \max \st_{s-1}(\alpha_n)$ and $x$ 
        never leaves $\st(\alpha_n)$ thereafter. Thus, by initialization 
        $\st(\alpha_n)\subseteq [p^{\alpha_n}_{fin}, \infty)$; 

    \item[(v)] at cofinitely many $\alpha_n$-stages we have $\alpha_n
        \subset f_s$.
\end{itemize}
Therefore, the infinite path $f=\bigcup_n \alpha_n$ (hence 
$f\!\upharpoonright\! n=\alpha_n$) is the desired infinite path 
through~$T$. 

We define $\alpha_n$ and $t_n$ by induction on $n$.

If $n = 0$ then take $\alpha_0 = \lambda$, and $t_0 = 0$. Clearly 
$\alpha_0$ and $t_0$ satisfy (i)--(iv) since if $s > 0$, then $s - 1 \in 
St_s(\lambda) \smallsetminus St_{s-1}(\lambda)$. Since we may assume that 
$\lambda$ is an $S$-strategy, then (v) holds as well for $\lambda$. 

Next, assume by induction that we have defined $\alpha_n$ and $t_n$ 
satisfying (i)--(v). Since the tree is finitely branching, and there are 
infinitely many stages $s$ such that $\alpha_n \subset f_s$, there is a 
string $\alpha_{n+1} \supset \alpha_n$ with infinitely many stages $s$ such 
that $\alpha_{n+1} \subset f_s$ and a last stage $t_{n+1} \geq t_n$ at 
which $\alpha_{n+1}$ is initialized.  This shows (i), (ii) and (iii) for 
$\alpha_{n+1}$ and $t_{n+1}$. 

Next, we show that  at cofinitely many $\alpha_{n+1}$-true stages $s> 
t_{n+1}$ we enumerate into $\st(\alpha_{n+1})$ some $x \geq 
p^{\alpha_{n+1}}_{fin}$ greater than all the elements currently in the 
stream. Indeed, if $\alpha_{n+1}=\alpha_n \tie \ang{o}$, with $o \in 
\set{\ocinf, \ocstop, \ocwait}$ then the claim easily follows by the 
inductive assumptions on $\alpha_n$, since, by construction at 
$\alpha_{n+1}$-true stages we have $\st(\alpha_{n+1})= \st(\alpha_n) \cap 
[p^{\alpha_{n+1}}_{fin}, \infty)$, and by induction greater and greater 
numbers are enumerated at  cofinitely many $\alpha_{n}$-true stages into 
$\st(\alpha_n)$. On the other hand, if $\alpha_{n+1}={\alpha_{n}} \tie 
\ang{\oci{i}}$ for some~$i$, then Case~2.2.3.3(a) applies.  The number $x$ 
which is added at any stage $s$ to $\st(\alpha_{n} \tie \ang{\oci{i}})$ in 
this case is in fact greater than or equal to $p^{\alpha_{n+1}}_s$ by 
construction, being greater than or equal to $p^{{\alpha_{n}} \tie 
\ang{\ocwait}}_s$. It follows that $\st(\alpha_{n+1}) \subseteq 
[p^{\alpha_{n+1}}_{fin},\infty)$ and hence (iv) holds. 

Finally, a stage is terminated prematurely only in Cases 2.2.1, 2.2.3.1, or 
2.2.3.2.  Case~2.2.1 can only happen finitely many times after $t_{n+1}$, 
since, as we have just seen, at cofinitely many $\alpha_{n+1}$-true stages 
after the last initialization of $\alpha_{n+1}$, a still available element 
greater than all the elements already in the stream enters the stream 
$\st(\alpha_{n+1})$. Therefore (thus showing (v) for $\alpha_{n+1}$) there 
exists a stage $t \geq t_{n+1}$ such that $\alpha_{n+1} \subset f_s$ at 
every $\alpha_{n+1}$-true stage $s \geq t$. Likewise, Cases~2.2.3.1 and 
2.2.3.2 can only happen once for $\alpha_{n+1}$ after $t_{n+1}$. 
\end{enumerate}
\end{proof}

\begin{definition}
At every stage $s$ of the construction, let $\widehat{\gamma}_s$ denote the 
maximum of any $D$-use in any $\Phi_k$ or $\Phi_e$ operator and any 
$\gamma$-marker defined up through and including stage $s$. 
\end{definition}

\begin{lemma}\label{lem:helper}
\begin{enumerate}
\item Every number $d$ can be chosen to be a $\gamma$-marker for some 
    $S$-strategy at most once.
    
\item If $F_{x_\alpha} = \set{d}$ in Case 2.2.3 of the construction, and 
    $d$ has not been chosen to be a $\gamma$-marker for any $S$-strategy, 
    then $d$ will never be chosen as a $\gamma$-marker.
     
\item If $d \in D_{s_0} \!\upharpoonright\! (\widehat{\gamma}_{s_0} + 1)$ 
    at some stage $s_0$ and $d$ is not a $\gamma$-marker at stage $s_0$, 
    either due to never being chosen or being canceled as a $\gamma$-maker 
    at some stage $s < s_0$, then $d \in D_s$ for all $s \geq s_0$. 
      
\item If $d$ is chosen to be a $\gamma$-marker for an $S$-strategy which is 
    later initialized at some stage $s_0$, then $D(d)[s] = D(d)[s_0]$ for 
    all $s \geq s_0$.
     
\item If $\beta \subset f_s$ is an $S$ strategy, the marker 
    $\gamma_\beta(z,x)$ is still defined at the end of stage $s$, and 
    $\gamma_\beta(z, x) \in D_s$, then $x \in A_s$. 
\end{enumerate}
\end{lemma}

\begin{proof} 
Since $\gamma$-markers are chosen to be larger than any number seen so far in 
the construction, and only defined $\gamma$-markers are extracted from $D$ 
due to the $\Gamma$-correction\big/$\Gamma$-killing actions of $R$ and $S$ 
strategies, (1)-(4) are immediate.  To show (5), first note that if $\beta$ 
saw $x \notin A$ at substage $|\beta|$ of stage $s$, then $\beta$ would 
cancel the $\gamma_\beta$-marker $\gamma_\beta(z,x)$ and extract it from $D$. 
Likewise, $\beta$ would only define and enumerate $\gamma_\beta(z,x)$ into 
$D$ at stage $s$ if it saw $x \in A$.  If at a later substage $x$ were 
extracted from $A$, as in Cases 2.2.3.1 or 2.2.3.2, then $\gamma_\beta(z, x)$ 
would be not be defined. 
\end{proof}

\begin{lemma} \label{lem:co-2-ce}
The set $A$  is $\Pi^0_1$ and the set $D$ is co-2-c.e.
\end{lemma}

\begin{proof}
First we show that $A$ is $\Pi^0_1$. Given a number $x$, we start with $x \in 
A$.  The number $x$ may be extracted from $A$ by an $R$-strategy $\alpha$ at 
a stage $s_0$ only via application of Case~2.2.3.1. But this can be done at 
most once since, after acting in Case~2.2.3.1, $x$ is never re-enumerated 
back into $A$ and consequently is no longer a suitable witness. 

We now show that $D$ is co-2-c.e.\ Given a number $d$, we start with $d \in 
D$.  By Lemma~\ref{lem:helper}(2), if $d$ is never chosen as a 
$\gamma$-marker, $d$ is never extracted from $D$.   If $d$ is chosen as a 
$\gamma$-marker, then $d$ can be extracted from $D$ for $\Gamma$-correction 
either in Case 1 or in Case 2.2.3.3(b), at which point $d$ is canceled as a 
$\gamma$-marker.  The number $d$ may then be enumerated back into $D$ via 
$\alpha$ in Case 2.2.3.2.  By Lemma~\ref{lem:helper}(3), $d$ cannot be 
extracted again. 
\end{proof}

Let $f$ be the true path whose existence is guaranteed by 
Lemma~\ref{lem:tree-lemma}(2). If $\alpha \subset f$ (and no reference to any 
specific stage of the construction is understood) then we denote by 
$\st(\alpha)$ the set of numbers enumerated in $\st_s(\alpha)$ at some stage 
after the last initialization of $\alpha$.

\begin{lemma}\label{lem:true-stage}
Let $\alpha \subset f$ be any strategy, and let $s_0$ be the least stage 
after which $\alpha$ is never initialized.  At every $\alpha$-true stage $s 
\geq s_0$, for all $x \in [0, s) \smallsetminus S(\alpha)$, $A_s(x) = A(x)$.  
Additionally, if $\alpha$ is an $R$-strategy and $\alpha \tie \ang{o} \subset 
f$ for $o \in \set{\ocstop, \ocwait}$, let $s_1 \geq s_0$ be the stage at 
which $\alpha$ first took on outcome $o$ with its final diagonalization 
witness $x_\alpha$.  Then $A_{s_1} \!\upharpoonright\! s_1 = A 
\!\upharpoonright\! s_1$. Furthermore, in the case that $o = \ocstop$, for 
every $s \geq s_1$, $(D_{s_1} \!\upharpoonright\! (\widehat{\gamma}_{s_1} + 
1) \smallsetminus \widehat{D}_{x_\alpha, s_1}) \cup F_{x_\alpha} \subset 
D_s$, where $\widehat{D}_{x_\alpha}$ is as defined in 
Definition~\ref{def:i-cleared} and $F_{x_\alpha}$ is as defined in either 
Case 2.2.3.1 or Case 2.2.3.2 at stage $s_1$. 
\end{lemma}

\begin{proof}
Let $\alpha \subset f$ and let $s_0$ be the least stage after which $\alpha$ 
is never initialized.  Choose $x \notin S(\alpha)$ and let $s_1$ be the least 
$\alpha$-true stage such that $s_1 \geq s_0$ and $s_1  > x$. Since $A$ is 
$\Pi^0_1$, if $x \notin A_{s_1}$, then $x \notin A$. Thus, we can assume $x 
\in A_{s_1}$.  No $S$-strategy can extract $x$ from $A$ and so we consider 
only $R$-strategies $\beta$. No strategy $\beta <_L \alpha$ is eligible to 
act after stage $s_0$ and so cannot extract $x$ from $A$. Every strategy 
$\beta 
>_L \alpha$ was initialized at stage $s_1$ and so any suitable witness chosen 
by such a $\beta$ after stage $s_1$ will be greater than $s_1 > x$, implying 
that $\beta$ cannot extract $x$ from $A$ after stage $s_1$.  Since $x \notin 
S(\alpha)$, no $\beta \supseteq \alpha$ can choose $x$ as a witness and so 
cannot extract $x$ from $A$. Finally, consider $\beta \subset \alpha$.	If 
$\beta \tie \ang{\ocstop} \subseteq \alpha$, then~$\beta$ does not extract 
any element from~$A$ after stage~$s_0$.  If $\beta \tie \ang{\ocwait} 
\subseteq \alpha$, then~$\beta$ cannot extract $x$ from~$A$ after stage $s_0$ 
without initializing $\alpha$, contrary to assumption. If $\beta \tie 
\ang{\oci{i}} \subseteq \alpha$, than any witness $y$ chosen by $\beta$  at 
any stage $s > s_1$ must be suitable for $\beta$, hence $y\ge p_s^{\beta \tie 
\ang{\ocwait}} \ge p_{s_1}^{\beta \tie \ang{\ocwait}}\ge s_1 > x$, and 
finally $y>x$. Similarly, $\beta$ cannot extract any element from $S(\beta 
\tie \ang{\oci{j}})$ for $j \leq i$ without initializing $\alpha$, while 
$\beta \tie \ang{\oci{j}}$ was initialized at stage $s_1$ and so any element 
of $S(\beta \tie \ang{\oci{j}})$ is greater than or equal to $s_1 > x$. 
Therefore, no strategy extracts $x$ from $A$ after stage $s_1$, and so $x \in 
A$. 

Now assume $o \in \set{\ocstop, \ocwait}$ and let $s_1 \geq s_0$ be the stage 
at which $\alpha$ first took on outcome $o$ with its final diagonalization 
witness $x_\alpha$.  The proof that $A_{s_1} \!\upharpoonright\! s_1 = A 
\!\upharpoonright\! s_1$ is analogous to the proof in the above paragraph, 
noting that $S_{s_1}(\alpha \tie \ang{o}) = \emptyset$, and that if $x_\alpha 
\in A_{s_1}$ then $o = \ocwait$ and so $x_\alpha$ cannot be extracted from 
$A$ without causing $\alpha \tie \ang{\ocwait}$ to be initialized, contrary 
to assumption. Finally, assume $o = \ocstop$ and let $s_1$, $x_\alpha$, 
$\widehat{D}_{x_\alpha, s_1}$ and $F_{x_\alpha}$ be as in the assumption of 
the lemma.  Consider $d \in (D_{s_1}\!\upharpoonright\! 
(\widehat{\gamma}_{s_1} + 1) \smallsetminus \widehat{D}_{x_\alpha, s_1}) \cup 
F_{x_\alpha}$.  By Lemma~\ref{lem:helper}(3), if $d$ is not a $\gamma$-marker 
at stage $s_1$, then $d \in D_s$ for all $s \geq s_1$.  So, we may assume 
that $d$ is a $\gamma_\beta$-marker at stage $s_1$ for some $S$-strategy 
$\beta$.  If $\beta <_L \alpha$, then~$\beta$ will not act after stage $s_1$, 
and so cannot extract~$d$ from~$D$.  If $\beta \geq \alpha\tie\ang{\ocstop}$, 
then~$\beta$ was initialized at stage $s_1$ and so by 
Lemma~\ref{lem:helper}(4) we have $d \in D$.  If $\beta \subset \alpha$, then 
$d = \gamma_\beta(z, x_d)$ for some $x_d \in A$ (when $d$ is appointed) by 
Lemma~\ref{lem:helper}(5), and $d = \gamma_\beta(z, x_d)$ was defined at some 
stage $s$ with $x_d < s \leq s_1$, and so (again by 
Lemma~\ref{lem:helper}(5)) $x_d \in A_{s_1} \!\upharpoonright\! s_1$.  This 
implies that $x_d \in A$ and thus $\beta$ will not extract $d$ from~$D$ due 
to $\Gamma$-correction. 

Now consider the possibility that $d$ is extracted from~$D$ due to the 
$\Gamma$-killing action of some $R$-strategy $\delta \supset \beta$.  Since 
$\delta$ only extracts $\gamma$-markers from~$D$ which were defined since 
$\delta$'s last initialization, necessarily we must have $\beta \subset 
\delta \subset \alpha$ and that for some $i$,  $\beta = \beta_i$ for 
$\delta$.  If $\delta \tie \ang{\oci{i}} <_L \alpha$, $\delta \tie 
\ang{\oci{i}}$ is not active after stage $s_1$, and so $\delta$ does not kill 
$\Gamma_\beta$ after stage $s_1$.  On the other hand, if $\delta \tie 
\ang{\oci{i}} \subseteq \alpha$ or $ \delta \tie \ang{\oci{i}} >_L \alpha$, 
then any $\Gamma_i$-killing action after stage $s_1$ will extract from $D$ 
only $\gamma_\beta$--markers defined after stage $s_1$, as it extracts only 
$\gamma_\beta$-markers which were picked up since the most recent stage at 
which either $\delta \tie \ang{\oci{i}}$ was last active or last initialized. 
\end{proof}

\begin{lemma}  
Every $R$-requirement is satisfied.
\end{lemma}

\begin{proof}
Fix an $R$-requirement $R_{k}$. By Lemma~\ref{lem:tree-lemma}(1) of the Tree 
Lemma, there exists a longest $\alpha \subset f$ such that $\alpha$ is 
assigned the requirement $R_k$, and for all $\beta$ with $\alpha \subset 
\beta \subset f$, $\alpha$ is satisfied along $\beta$. Fix such an $\alpha 
\subset f$.  Then $\alpha \tie \ang{o} \subset f$ for $o \in \set{\ocstop, 
\ocwait}$. By Lemma~\ref{lem:tree-lemma}(2)(c), $\alpha$ eventually picks 
from the numbers which are suitable witnesses for it, a witness which becomes 
its eventual diagonalization witness. Let $s_0$ be the least stage after 
which $\alpha$ is never initialized and has picked its final diagonalization 
witness $x_\alpha$. 

If $\alpha \tie \ang{\ocwait} \subset f$, then at co-finitely many 
$\alpha$-true stages $s > s_0$, $x_\alpha \in A_s \smallsetminus 
\Phi_k(D)[s]$. Since $A \in \Pi^0_1$ (Lemma~\ref{lem:co-2-ce}), $\alpha$ 
never extracts $x_\alpha$ from $A$. On the other hand, $x_\alpha \notin 
\Phi_k(D)$, and so $R_k$ is satisfied. 

If $\alpha \tie \ang{\ocstop} \subset f$, let $s_1 \geq s_0$ be the 
$\alpha$-stage at which Case~2.2.3 applied and $x_\alpha$ became a realized 
witness via some axiom $\ang{x_\alpha, F_{x_\alpha}} \in \Phi_k$ with 
$F_{x_\alpha} \subseteq D$.  Let $s_2 \geq s_1$ be the $\alpha$-stage at 
which Case~2.2.3.1 or Case~2.2.3.2 applied.  By Lemma~\ref{lem:true-stage}, 
$F_{x_\alpha} \subseteq D_s$ for all $s \geq s_2$, ensuring that $x_\alpha 
\in \Phi_k(D)[s]$ for all $s \geq s_2$.  Likewise, by the action of the 
construction in Cases 2.2.3.1 and Case 2.2.3.2, $x_\alpha \notin A_{s_2}$ 
and, again by Lemma~\ref{lem:co-2-ce}, $x_\alpha \notin A_s$ for all $s \geq 
s_2$.  Thus $x_\alpha \in \Phi_k(D) \smallsetminus A$. 
\end{proof}

\begin{lemma}\label{lem:S-satifaction2}
Every $S$-requirement is satisfied.
\end{lemma}

\begin{proof}
Fix an $S$-requirement $S_e$.  By the Tree Lemma 
(Lemma~\ref{lem:tree-lemma}), there exists a longest $\beta \subset f$ such 
that $\beta$ is assigned the requirement $S_e$ and $\beta$ is either active 
along $\alpha$ for all $\alpha$ with $\beta \subset \alpha \subset f$, or 
there is a minimal $\alpha$ with $\beta \subset \alpha \subset f$ and $\beta$ 
is active along all $\delta$ with $\beta \subseteq \delta \subseteq \alpha$ 
and is satisfied along all $\delta$ with $\alpha \subset \delta \subset f$.  
Fix such a $\beta \subset f$ and the last stage $s_0$ at which $\beta$ was 
initialized.  We consider both cases.
 
\emph{Case 1}: $\beta$ is active along $\alpha$ for all $\alpha$ with $\beta 
\subset \alpha \subset f$. We show that in this case $\Phi_e(A \oplus B) = 
\Gamma_\beta(D)$, where $\Gamma_\beta$ is the $s$-operator consisting of all 
the axioms which were enumerated into $\Gamma_\beta$ after stage $s_0$.	By 
construction, $\Phi_e(\emptyset \oplus D) \subseteq \Gamma_\beta(D)$.  Since 
$\Phi_e(A \oplus D) = \Phi_e(A \oplus \emptyset) \cup \Phi_e(\emptyset \oplus 
D)$,  we only need to show that $\Phi_e(A \oplus \emptyset) \smallsetminus 
\Phi_e(\emptyset \oplus D)  = \Gamma_\beta(D) \smallsetminus \Phi_e(\emptyset 
\oplus D)$.

First, choose $z \in \Phi_e(A \oplus \emptyset) \smallsetminus 
\Phi_e(\emptyset \oplus D)$.  Consider an axiom $\ang{z, \set{x} \oplus 
\emptyset} \in \Phi_e$ with $x \in A$. Let $s_1 \geq s_0$ be the least 
$\beta$-true stage at which $\ang{z, \set{x} \oplus \emptyset} \in \Phi_e$.  
Since $x \in A$ and $A$ is $\Pi^0_1$, at every $\beta$-true stage $s \geq 
s_1$, $z \in \Phi_e(A \oplus \emptyset)$.  Thus, at every $\beta$-true stage 
$s \geq s_1$, if $\gamma_\beta(z, x)$ is undefined, $\beta$ will define a new 
$\gamma_\beta(z, x)$-marker, enumerate $\gamma_\beta(z, x)$ into $D$, and 
enumerate the axiom $\ang{z, \set{\gamma_\beta(z, x)}}$ into $\Gamma_\beta$. 	

Since $x \in A$, $\beta$ will never extract $\gamma_\beta(z,x)$ from $D$.  
Thus $\gamma_\beta(z,x)$ can only be extracted from $A$ by an $R$-strategy 
$\alpha$ where $\beta \subset \alpha$, $|\alpha| < x$, $\beta$ is active 
along $\alpha$ (implying that $\beta$ is $\beta_i$ for $\alpha$ as in Case 2 
of the construction), and only if $\alpha \tie \ang{\oci{j}}$ is active for 
some $j \leq i$ and $\gamma_\beta(z, x)$ was chosen after $\alpha$'s last 
initialization and the last stage at which $\alpha \tie \ang{\oci{j}}$ was 
active: see Case 2.2.3.3(b) of the construction. 	

Let $s_2 \geq s_1$ be the least $f \!\upharpoonright\! x$-true stage after 
which $f \!\upharpoonright\! x$ is not initialized.  If $\gamma_\beta(z, x)$ 
is not defined at the beginning of stage $s_2$, $\beta$ will define a new 
$\gamma_\beta(z, x)$-marker at stage $s_2$.  We show that the 
$\gamma_\beta(z, x)$-marker which is defined, or already defined, at stage 
$s_2$ is never canceled or extracted from $D$: in the following, we examine, 
and exclude one by one, all possible cases that could lead to this 
extraction. No $\alpha <_L f \!\upharpoonright\! x$ will act after stage 
$s_2$ and so cannot extract $\gamma_\beta(z, x)$ from $D$.  Every $\alpha >_L 
f \!\upharpoonright\! x$ is initialized at stage $s_2$ and so no such 
$\alpha$ will be allowed to extract $\gamma_\beta(z, x)$ from $D$ at any 
$\alpha$-true stage $s > s_2$.  Any $\alpha \supset f \!\upharpoonright\! x$ 
cannot extract $\gamma_\beta(z, x)$ from $D$ since $|\alpha| > x$. We finally 
examine the case when $\beta$ is active along all $\alpha$ with $\beta 
\subset \alpha \subset f \!\upharpoonright\! x$. For any such $\alpha$, if 
$\alpha$ is an $R$-requirement then $\beta = \beta_i$ for some $i$.  The 
$\gamma_\beta$-marker $\gamma_\beta(z, x)$ can only be canceled and extracted 
from $D$ if $\alpha \tie \ang{\oci{j}}$ is active after stage $s_2$ for some 
$j \leq i$.  Since $\beta$ is active along $\alpha \tie \ang{o} \subseteq f 
\!\upharpoonright\! x$, it must be that $o \in \set{\ocstop, \ocwait} \cup 
\set{\oci{k}:k > i}$.  If $o = \ocstop$, then since $\alpha \tie 
\ang{\ocstop}$ is eligible to act at stage $s_2$, no strategy $\alpha \tie 
\ang{o'}$ for $o' \neq \ocstop$ is eligible to act after stage $s_2$.  If $o 
\neq \ocstop$, then no $\alpha \tie \ang{\oci{j}}$ for $j \leq i$ is eligible 
to act after stage $s_2$ as this would cause $\alpha \tie \ang{o}$ to be 
initialized, contrary to assumption. 		

Now choose $z \in \Gamma_\beta(D) \smallsetminus \Phi(\emptyset \oplus D)$.  
Then $z \in \Gamma_\beta(D)$ via some axiom $\ang{z, \set{d}} \in 
\Gamma_\beta$ with $d \in D$ where $\ang{z, \set{d}}$ was enumerated into 
$\Gamma_\beta$ at some stage $s \geq s_0$ at which $d$ was a 
$\gamma_\beta$-marker. We have two cases to consider. 		

\emph{Case 1.1}: There exists an axiom $\ang{z, \set{d}} \in \Gamma_\beta$ 
with $d \in D$ such that the $\gamma_\beta$-marker $d$ was never canceled.  
Let $s_1 \geq s_0$ be the $\beta$-true stage at which $d = \gamma_\beta(z, 
x)$ for some $x$ was defined. Thus, $\ang{z, \set{x} \oplus \emptyset} \in 
\Phi_{e, s_1}$ and $x \in A_{s_1}$.  This implies that $x \in A_{s}$ for 
every $s \ge s_1$, otherwise since $A \in \Pi^0_1$ there would exist a stage 
$s_2\ge s_1$ such that $x \notin A_s$ for every $s > s_2$, and by Case~1 of 
the construction at some $\beta$-true stage $s >s_2$ we would cancel the 
$\gamma_\beta$-marker $d$. Therefore $z \in \Phi(A\oplus \emptyset)$ as $x 
\in A$.		

\emph{Case 1.2}: For all axioms $\ang{z, \set{d}} \in \Gamma_\beta$ with $d 
\in D$, the $\gamma_\beta$-marker $d$ was canceled.  Fix a particular such 
axiom $\ang{z, \set{d}}$ and let $s_1 \geq s_0$ be the $\beta$-true stage at 
which $d$ was defined.  The $\gamma_\beta$-marker $d$ can only be canceled in 
Case 2.2.3.1, at which point it is restrained in $D$, or canceled and 
extracted from $D$ in Case 2.2.3.3 and then enumerated back into $D$ in Case 
2.2.3.2, by some $R$-strategy $\alpha \supset \beta$ where $\beta$ is active 
along $\alpha$.  Let $s_2 \geq s_1$ be the $\alpha$-true stage at which 
either Case 2.2.3.1 or Case 2.2.3.2 applied.  In either case, $\alpha$ has a 
defined diagonalization witness $x_\alpha$ with $F_{x_\alpha} = \set{d}$ 
which is $\Gamma_i$-cleared for all $\beta_i \subset \alpha$ which are active 
along $\alpha$.  In particular, $x_\alpha$ is $\Gamma_\beta$-cleared, which 
means that $z \in \Phi_e((A \smallsetminus \set{x_\alpha}) \oplus ((D 
\smallsetminus \widehat{D}_{x_\alpha}) \cup F_{x_\alpha})[s_2]$.  By 
Lemma~\ref{lem:true-stage}, this implies that $z \in \Phi_e(A \oplus D)$. 

\emph{Case 2}: There is a minimal $\alpha$ with $\beta$ active along all 
$\delta$ with $\beta \subseteq \delta \subseteq \alpha$ and satisfied along 
all $\delta$ with $\alpha \subset \delta \subset f$. 	

In this case, $\alpha$ is an $R$-strategy, $\beta$ is $\beta_i$ for $\alpha$ 
and $\alpha\tie\ang{\oci{i}} \subset f$.  We note that the operator $\Phi_e$ 
of $S$ is referred to as $\Phi_i$ by $\alpha$.  We show that 
$\Delta_i(\Phi_i(A \oplus D)) = A$ where $\Delta_i$ is the $s$-operator 
consisting of the axioms enumerated by $\alpha$ into $\Delta_i$ after the 
last initialization of $\alpha \tie \ang{\oci{i}}$.  Let $\beta_0 \subset 
\cdots \subset \beta_{n - 1} \subset \alpha$ be the set of $S$-strategies 
which are active along $\alpha$. Let $s_1$ be the last stage at which $\alpha 
\tie \ang{\oci{i}}$ was initialized. 

First, consider $x \in \omega \smallsetminus S(\alpha \tie \ang{\oci{i}})$.  
Let $s_2$ be the least $\alpha \tie \ang{\oci{i}}$-true stage such that $s_2 
> s_1, x$.  At no $\alpha \tie \ang{\oci{i}}$-true stage $s'_2$ with $s_1 
\leq s'_2 < s_2$, did $\alpha$ enumerate any axiom of the form $\ang{x, F}$ 
into $\Delta_i$. By Lemma~\ref{lem:true-stage}, $A_{s_2}(x) = A(x)$, and 
since $A$ is $\Pi^0_1$, at every stage $s \geq s_2$, $A_s(x) = A(x)$. If $x 
\in A_{s_2}$, then the axiom $\ang{x, \emptyset}$ is enumerated into 
$\Delta_i$ at stage $s_2$, giving $x \in \Delta_i(\Phi_i(A \oplus D))$.  On 
the other hand, if $x \notin A_{s_2}$, then at no stage $s \geq s_2$ do we 
see $x \in A_{s}$, and so $\alpha$ never enumerates an axiom of the form 
$\ang{x, F}$ into $\Delta_i$.  Thus $x \notin \Delta_i(\Phi_i(A \oplus D))$, 
showing that $A_{s_2}(x) = A(x) = \Delta_i(\Phi_i(A \oplus D))(x)$. 

Finally, consider $x \in S(\alpha \tie \ang{\oci{i}})$. At the stage $s_2 
> s_1$ at which $x$ was enumerated into $S(\alpha \tie \ang{\oci{i}})$ we 
saw $x \in \Phi_k(D)$ via an axiom $\ang{x, \set{d}} \in \Phi_k$ with $d \in 
D$ where $d = \gamma_i(z_x, x)$ was a $\gamma_i$-marker defined by $\beta_i$, 
and so we enumerated the axiom $\ang{x, \set{z_x}}$ into $\Delta_i$. Thus, if 
$x \in A$, then $z_x \in \Phi_i(A \oplus D)$ and so $x \in \Delta_i(\Phi_i(A 
\oplus D))$. Assume $x \notin A$. Then $x$ was extracted from $A$ at some 
stage $s_3 \geq s_2$. At the next $\alpha$-true stage $s_4 \geq s_3$, all the 
$\beta_j$-strategies will have canceled any $\gamma_j(\cdot, x)$-markers and 
extracted them from $D$. Furthermore, no new $\gamma_j(\cdot, x)$-markers 
will be defined after stage $s_2$. Thus, $\widehat{D}_x[s] = \emptyset$ for 
all stages $s \ge s_4$. Since $x$ is not $\Gamma_i$-cleared (as otherwise $x$ 
could be used as the final diagonalization witness of $\alpha$, which would 
initialize $\alpha \tie \ang{\oci{i}}$, contrary to assumption), we have 
\[
z_x \notin \Phi_i((A \smallsetminus \set{x}) \oplus ((D \oplus 
\widehat{D}_x) \cup \set{d}))[s] = \Phi_i(A \oplus (D \cup \set{d}))[s] 
\supseteq \Phi_i(A \oplus D)[s]
\]
for all $\alpha$-true stages $s \ge s_4$. Hence $z_x \notin \Phi_i(A \oplus 
D)$ and thus $x \notin \Delta_i(\Phi_i(A \oplus D))$. 
\end{proof}

The proof of Theorem~\ref{thm:delta2emptyintervals} is now complete.
\end{proof}

\begin{corollary}
The $\Delta^0_2$ $s$-degrees are not dense. The $\Delta^0_2$ $Q$-degrees are 
not dense. 
\end{corollary}

\begin{proof}
Immediate.
\end{proof}

\section{Finding the two sets in the same $e$-degree}

Finally we show that in the previous theorem, the $s$-degrees $\bs{d}$ and 
$\bs{e}$ can be built in the same $e$-degree, by showing that we can build 
$A$ and $D$ so that $A=\Omega(D)$, where $\Omega$ is a $e_2$-operator, i.e. 
an $e$-operator such that, for every number $x$ and finite set $G$, 
\[
\ang{x, G} \in \Omega \Rightarrow [\textrm{$G= \emptyset$
or $G$ has exactly two elements].}
\]

Recall that $\set{F_n}_{n \in \omega}$ is a strong disjoint array if there is 
a computable function $f$ such that for each $n$, $f(n)$ gives the canonical 
index of $F_n$ and the $F_n$ are pairwise disjoint.  

\begin{corollary}\label{cor:be}
The sets $A$ and $D$ of Theorem~\ref{thm:delta2emptyintervals} can be built 
such that $D \equiv_{e} A \oplus D$. 
\end{corollary}

\begin{proof}
Let $\set{F_n}_{n \in \omega}$ be a strong disjoint array of nonempty finite 
sets which partition $2\omega = \set{2x : x \in \omega}$, where $\card{F_n} = 
2$ for all $n$ (the symbol $\card{X}$ denotes the cardinality of any given 
set $X$): for instance take $F_n=\{4n, 4n+2\}$.  Define the $e_2$ operator 
$\Omega = \set{\ang{n, F_n}:n \in \omega}$.  We construct a co-2-c.e. set $D$ 
and take $A = \Omega(D)$.  If we ever need to extract an element $n$ from 
$A$, we extract a single element $F_n$ from $D$, but never both.  

We then run the construction in the same way as in the proof of 
Theorem~\ref{thm:delta2emptyintervals} (from which we borrow terminology and 
notations), starting with $D = \omega$ (and thus $A = \omega$), with the 
following modifications. 
\begin{itemize}
\item The $\gamma$-markers must be chosen from $D \cap (2\omega +1)$ (where 
    $2\omega +1=\set{2x+1: x \in \omega}$)). This ensures that extractions 
    from $D$ (for the sake of $\Gamma$-corrections, and $\Gamma$-killing) 
    do not cause extractions from $A$. 
\item When a witness $x$ becomes realized for the $R$-requirement $A\ne 
    \Phi(D)$, then a strategy $\alpha$ for this requirement prescribes the 
    action of extracting $x$ from $A$ and restraining $x \in \Phi(D)$. If 
    $x$ has become realized via an axiom $\ang{x, \set{d} }\in \Phi$ with 
    $d$ even then we extract from $D$ exactly one of $F_x \smallsetminus 
    \set{d}$, winning the requirement as we get $x\in \Phi(D) 
    \smallsetminus A$, and $x$ is $\Gamma_i$-cleared for all $i<n$. (As 
    usual, we are assuming that $\beta_0\subset \beta_1 \subset \cdots 
    \subset \beta_{n-1}$ are the $S$-strategies which are active along 
    $\alpha$.) If $d$ is odd and $d$ is a $\gamma$-marker then the action 
    taken by the construction depends on whether $x$ is $\Gamma_i$-cleared 
    for all $i<n$, or not.  In any case, the details are the same as in the 
    proof of the previous theorem, modulo obvious adaptations. 
\end{itemize}
\end{proof}

\section{Questions}

Can it be argued that Theorem~\ref{thm:delta2emptyintervals} provides an 
almost optimal nondensity result for the local structure of the $s$-degrees? 
The sets $D$ and $E$ provided by the theorem are co-$2$-c.e.\,. We cannot go 
from co-$2$-c.e. sets to both $2$-c.e. sets, as by 
\cite[Lemma~19]{Omanadze-Sorbi:s-degrees} every $2$-c.e. set is 
$s$-equivalent to a $\Pi^0_1$ set, and Downey, La~Forte and 
Nies~\cite{Downey-Laforte-Nies:Q-degrees} proved that the $\Sigma^0_1$ 
$Q$-degrees are dense, which is equivalent to saying that the $\Pi^0_1$ 
$s$-degrees are dense. 

Therefore the following question remains open: 

\begin{questionrm}
Can one find suitable co-$2$-c.e.\, sets $D,E$ witnessing nondensity as in 
Theorem~\ref{thm:delta2emptyintervals}, but one of them is $\Pi^0_1$? 
\end{questionrm}

The construction provided in the proof of  
Theorem~\ref{thm:delta2emptyintervals} does not work of course in the 
$e$-degrees since it is known that the $\Sigma^0_2$ $e$-degrees are dense.  
So, the proof of Theorem~\ref{thm:delta2emptyintervals} makes an unavoidable 
use of the special form of the $s$-operators within the full class of the 
$e$-operators: whereas in $s$-reducibility the use of a computation $z\in 
\Phi(A\oplus D)$ comes from either $A$ or $D$, the use in $e$-reducibility 
may involve both sets, and more elements for each set. This special form of 
the $s$-operators is exploited in many spots of the construction: it appears 
for example in the definition of a setup $(z,x)$ for $\Delta_i$ as defined in 
Case~2.2.3.3 of the construction, where it can be correctly assumed that $z_x 
\in \Phi_i(A\oplus D)$ only by keeping $x\in A$, without needing to restrain 
any other number in $A$, or any number in $D$. 

It would be interesting to spot where nondensity of the $\Sigma^0_2$ degrees 
gets lost when one moves from $s$-reducibility towards $e$-reducibility by 
gradually enlarging the class of $e$-operators used for the reductions. For 
instance, Cooper (\cite[Question~\S5(e)]{Cooper:Enumeration} and 
\cite[Question~26]{Cooper:Enumeration}) asks whether the $\Sigma^{0}_{2}$ 
$be$-degrees are dense. The reducibility known as $be$-reducibility is 
defined in terms of $be$-operators: An $e$-operator $\Phi$ is a 
\emph{$be$-operator}, if there exists a number $n$, such that 
\[
(\forall x)(\forall D)\left[\ang{x, D} \in \Phi \Rightarrow \card{D}\leq n\right].
\]
Clearly every $s$-operator is a $be$-operator. The $e$-operator $\Omega$ 
built in Corollary~\ref{cor:be} is an example of a $be$-operator with $n=2$. 


\end{document}